\documentclass[11pt,a4paper,reqno]{amsart}
\usepackage{amsmath}
\usepackage{amssymb}
\usepackage{euscript}
\usepackage{pdfsync}
\usepackage[all]{xy}
\usepackage{color}
\usepackage{setspace}
\usepackage[colorlinks=true,linktocpage=true,pagebackref=false, citecolor=black,linkcolor=black]{hyperref}
\usepackage{cancel}
\usepackage[normalem]{ulem}
\usepackage{tikz,pgfplots}
\pgfplotsset{compat=1.10}
\usepgfplotslibrary{fillbetween}
\pgfplotsset{soldot/.style={color=black,only marks,mark=*}} \pgfplotsset{holdot/.style={color=black,fill=white,only marks,mark=*}}

\newtheorem{thm}{Theorem}[section]

\newtheorem{lem}[thm]{Lemma}
\newtheorem{prop}[thm]{Proposition}

\newtheorem{cor}[thm]{Corollary}

\newtheorem{defn}[thm]{Definition}
\newtheorem{defns}[thm]{Definitions}

\newtheorem{remark}[thm]{Remark}

\newtheorem{example}[thm]{Example}
\newtheorem{examples}[thm]{Examples}

{\em}

{\em}

\newcounter{substep}
\def\thesubstep{\arabic{substep}}

{\em}

\newenvironment{point}[1]{%
\refstepcounter{thm}\noindent{\bf (\thethm)\quad {#1}}\ }%
{}

\newcommand{\Z}{{\mathbb Z}} \newcommand{\R}{{\mathbb R}}
 \newcommand{\C}{{\mathbb C}}

\newcommand{\sph}{{\mathbb S}} 
 \newcommand{\Proy}{{\mathbb P}}

 \newcommand{\J}{{\mathcal J}}

\newcommand{\gtp}{{\mathfrak p}} \newcommand{\gtq}{{\mathfrak q}}
\newcommand{\gtm}{{\mathfrak m}} \newcommand{\gtn}{{\mathfrak n}}
\newcommand{\gta}{{\mathfrak a}} \newcommand{\gtb}{{\mathfrak b}}

\newcommand{\Dd}{{\EuScript D}}
\newcommand{\Zz}{{\EuScript Z}}
\newcommand{\Uu}{{\EuScript U}}

\newcommand{\Bb}{{\EuScript B}}

\newcommand{\im}{\operatorname{im}}

\newcommand{\Int}{\operatorname{Int}}
\newcommand{\dist}{\operatorname{dist}}

\newcommand{\Max}{\operatorname{Max}}
\newcommand{\Spec}{\operatorname{Spec}}

\newcommand{\cl}{\operatorname{Cl}}

\newcommand{\Specs}{\operatorname{Spec_s}}

\newcommand{\betas}{\operatorname{\beta_s\!}}

\newcommand{\Min}{\operatorname{Min}}
\newcommand{\cd}{\operatorname{Card}}

\newcommand{\x}{{\tt x}}  
 \renewcommand{\t}{{\tt t}}

\newcommand{\veps}{\varepsilon}

\newcommand{\ol }{\overline}

\begin{document}

\title[Locally injective semialgebraic mappings]{Locally injective  semialgebraic mappings}

\author{E. Baro}
\author{Jos\'e F. Fernando}
\author{J.M. Gamboa}

%    \date is required; it is the date received by the editor.
%\date{October 21, 2008}
\thanks{Authors supported by Spanish STRANO PID2021-122752NB-I00 and Grupos UCM 910444}

\begin{abstract}
We characterize locally injective semialgebraic maps between two semialgebraic sets in terms of the induced homomorphism between their rings of (continuous) semialgebraic functions. 
\end{abstract}

\maketitle

% \date is required; it is the date received by the editor.
%\date{October 21, 2008}

\maketitle
\setcounter{tocdepth}{2}
{\small
\begin{spacing}{0.01}
\tableofcontents
\end{spacing}
}

\section{Introduction}\label{s1}

Locally injective maps have deserved the attention of mathematicians for a long time. Indeed, the classical inverse mapping theorem is a result about locally injective maps, but some work has been done to prove the  local injectivity of continuous maps. In that vein it is worthwhile mentioning Massey article \cite{m}. As a consequence of his main results 2.3 and 2.4  he proved in Corollary 6 that if $f:\R^n\to\R^n$ is a continuous map  and $U\subset\R^n$  is an open subset whose closure $\cl_{\R^n}(U)$ is compact and whose boundary $\partial U:=\cl_{\R^n}(U)\setminus U$ is a connected and $(n-1)$-dimensional orientable manifold such that $f|_U$ is a local homeomorphism and $f|_{\partial U}$ is injective, then $f|_{\cl_{\R^n}(U)}:\cl_{\R^n}(U)\to f(\cl_{\R^n}(U))$ is a homeomorphism.

On the other hand, in the article \cite{l} by Lewenberg also a global property is deduced from local injectivity; he proves that locally injective maps in o-minimal structures without poles are surjective.

In this paper we characterize  locally injective semialgebraic maps $\pi:M\to N$ between semialgebraic sets $M$ and $N$ in terms of some finiteness properties of the induced homomorphism $\varphi_{\pi}:{\mathcal S}(N)\to{\mathcal S}(M),\, f\mapsto f\circ\pi$ between their rings of continuous semialgebraic functions ${\mathcal S}(N)$ and ${\mathcal S}(M)$. 

To lighten notations we denote $\Specs(M):=\Spec\,({\mathcal S}(M))$ the set of prime ideals of ${\mathcal S}(M)$ endowed with the Zariski topology. Our previous articles \cite{bfg}, \cite{f}, \cite{f1}, \cite{f2}, \cite{fg3}, \cite{fg4}, \cite{fg6}, \cite{fg5}, \cite{fg1} and \cite{fg2} are devoted to study the relationship between $\pi$ and its spectral counterpart 
$$
\Specs(\pi):\Specs(M)\to\Specs(N),\, \gtp\mapsto\varphi_{\pi}^{-1}(\gtp),
$$ 
and this article is a new step in our attempt to understand this relationship. 

A subset $M\subset\R^m$ is said to be \textit{basic semialgebraic} if it can be written as
$$
M:=\{x\in\R^m:\ f(x)=0,\, g_1(x)>0,\ldots,g_{\ell}(x)>0\}
$$
for some polynomials $f,g_1,\ldots,g_{\ell}\in\R[\x_1,\ldots,\x_m]$. The finite unions of basic semialgebraic sets are called \emph{semialgebraic sets}. A continuous function $f:M\to\R$ is said to be \emph{semialgebraic} if its graph is a semialgebraic subset of $\R^{n+1}$. Usually, semialgebraic function just means a function, non necessarily continuous, whose graph is semialgebraic. However, since all semialgebraic functions occurring in this article are continuous we will omit for simplicity the continuity condition when we refer to them. Likewise, a continuous mapping $\pi:M\to N$ between semialgebraic sets whose graph is semialgebraic will be called, simply, a \em semialgebraic mapping.\em 

The sum and product of functions, defined pointwise, endow the set ${\mathcal S}(M)$ of semialgebraic functions on $M$ with a natural structure of commutative ring whose unity is the semialgebraic function ${\bf 1}_M$ with constant value 1. In fact ${\mathcal S}(M)$ is an $\R$-algebra, if we identify each real number $r$ with the constant function which just attains this value. The most simple examples of semialgebraic functions on $M$ are the restrictions to $M$ of polynomials in $n$ variables. Other relevant ones are the absolute value of a semialgebraic function, the distance function to a given semialgebraic set, the maximum and the minimum of a finite family of semialgebraic functions, the inverse and the $k$-root of a semialgebraic function whenever these operations are well-defined.

Each semialgebraic map $\pi:M\to N$ between semialgebraic sets $M$ and $N$ induces a homomorphism of $\R$-algebras $\varphi_{\pi}:{\mathcal S}(N)\to{\mathcal S}(M),\, g\mapsto g\circ\pi$, and in this work we characterize the local injectivity of $\pi$ in terms of finiteness conditions imposed to this homomorphism. 

Note that each homomorphism $\varphi:{\mathcal S}(N)\to{\mathcal S}(M)$ of rings with unity is a homomorphism of $\R$-algebras. To prove this observe that $\varphi({\bf 1}_N)={\bf 1}_M$, so given $r\in\R$ and a point $x\in M$ we have $\varphi({\bf 1}_N)(x)={\bf 1}_M(x)=1$. Hence $\varphi_x:\R\to\R, \, r\mapsto \varphi(r)(x)$ is a field homomorphism with $\varphi_x(1)=1$. Therefore $\varphi_x$ is the identity, that is, $\varphi(r)(x)=r$ for every $x\in M$ and $r\in\R$, that is, $\varphi(r)=r$. 

In Section 2 we introduce the terminology employed along the paper, whereas the relationship between finiteness conditions of $\pi$ and $\varphi_{\pi}$ are collected in Section 3. A semialgebraic map $\pi:M\to N$ is said to be \em locally injective \em if there exists a family ${\mathcal U}:=\{U_i:\, i\in I\}$ of open semialgebraic subsets of $M$ such that each restriction $\pi|_{U_i}:U_i\to N$ is injective. We will see in Remark \ref{FNTSS} that in case $\varphi_{\pi}$ is a finite homomorphism, then $\pi$ is locally injective and the family ${\mathcal U}$ above can be assumed to be finite. However, local injectivity does not implies injectivity. For example, $\pi:\R\to\R^2,\, t\mapsto(t^2-1,t(t^2-1))$ is a non injective but locally injective semialgebraic map. 

The main result of the paper is Theorem \ref{lciny}, where we prove that if $\varphi_{\pi}:{\mathcal S}(N)\to{\mathcal S}(M)$ is a finite homomorphism, then the maps
$$
\Specs(\pi):\Specs(M)\to\Specs(N)\text{ and }\pi:M\to N
$$
are proper, separated, locally injective and their fibers are finite sets. As a consequence we deduce in Corollary \ref{chliy} that if $M$ and $N$ are compact semialgebraic sets, a semialgebraic map $\pi:M\to N$ is locally injective if and only if $\varphi_{\pi}:{\mathcal S}(N)\to{\mathcal S}(M)$ is finite. 

\section{Preliminaries}\label{s2}

Let us fix some notations and recall some basic results concerning semialgebraic sets and functions.  For each $f\in{\mathcal S}(M)$ and each semialgebraic subset $N\subset M$, we denote
$$
\Zz_N(f):=\{x\in N:\, f(x)=0\}\quad\&\quad\Dd_N(f):=N\setminus\Zz_N(f).
$$ 
In addition, if $N\subset M$ we denote $\Int_M(N)$ and $\cl_M(N)$, respectively, the interior and the closure of $N$ in $M$.

\begin{lem} \label{tietze} Let $N$ be a closed semialgebraic subset of a semialgebraic set $M$.

\noindent \em (1) \em There exists a function $f\in{\mathcal S}(M)$ such that $N=\Zz_M(f)$. 

\noindent \em (2) \em The restriction homomorphism ${\mathcal S}(M)\to{\mathcal S}(N),\, f\mapsto f|_N$ is surjective.
\end{lem}

\begin{proof} (1) It suffices to choose $f:=\dist(\cdot,N)$.

\noindent (2) This was proved by Delfs and Knebusch in \cite[Thm. 3]{dk}. 
\end{proof}

For our purposes it is useful to know that ${\mathcal S}(M)$ is a \em Gelfand ring, \em that is, each prime ideal in ${\mathcal S}(M)$ is contained in a unique maximal ideal. This was proved  by Carral-Coste \cite{cc} and Schwartz \cite{s}. We prove in Proposition \ref{Gelfand} a slightly stronger result. First we need an auxiliary lemma.

\begin{lem}\label{quot1} Let $M$ be a semialgebraic set and let $f,g\in{\mathcal S}(M)$ with $\Zz_M(g)\subset\Int_M(\Zz_M(f))$. Then, there exists $h\in{\mathcal S}(M)$ with $f=gh$ and $\Zz_M(f)\subset\Zz_M(h)$. 
\end{lem}
\begin{proof} The open semialgebraic subsets $U:=M\setminus\Zz_M(g)$ and $V:=\Int_M(\Zz_M(f))$ of $M$ cover $M$ and $f|_{U\cap V}\equiv0$. Thus, 
$$
h:M\to\R,\ x\mapsto\left\{
\begin{array}{cc}
\frac{f(x)}{g(x)}&\text{if $x\in U$,}\\[4pt]
0&\text{if $x\in V$,}\\
\end{array}
\right.
$$
is a semialgebraic function satisfying $f=gh$ and $\Zz_M(f)\subset\Zz_M(h)$.
\end{proof}

\begin{prop} \label{Gelfand}\em (1) \em Let $M$ be a semialgebraic set and let  $\gta_1,\gta_2$ be two comaximal ideals in ${\mathcal S}(M)$. Then, its intersection $\gta_1\cap\gta_2$ does not contain any prime ideal in ${\mathcal S}(M)$.

\noindent \em (2) \em ${\mathcal S}(M)$ is a Gelfand ring. 
\end{prop}
\begin{proof} The second part follows from the first one because two distinct maximal ideals are comaximal. For the first part suppose, by the way of contradiction, the existence of a prime ideal $\gtp$ in ${\mathcal S}(M)$ contained in $\gta_1$ and $\gta_2$. Consider the semialgebraic functions
$$
G_1:\R\to\R, \, t\mapsto\left\{
\begin{array}{cl}
0&\text{if $|t|\geq2$,}\\[4pt]
t+2&\text{if $-2<t<-1$,}\\[4pt]
1&\text{if $-1\leq t\leq1$,}\\[4pt]
2-t&\text{if $1<t<2$,}\\
\end{array}
\right.\quad\&\quad H:\R\to\R,\, t\mapsto\left\{
\begin{array}{cl}
0&\text{if $t\leq1/3$,}\\[4pt]
6t-2&\text{if $2<6t<3$,}\\[4pt]
1&\text{if $t\geq1/2$,}\\
\end{array}
\right.
$$
and define $G_2\in{\mathcal S}(\R)$ by $G_2(t)=G_1(t+4)$. Notice that the product  $G_1G_2\equiv0$ and define $h:=5H-4\in{\mathcal S}(\R)$.

The ideals $\gta_1$ and $\gta_2$ being comaximal, there exist $f_1\in\gta_1$ and $f_2\in\gta_2$ such that $1=f_1+f_2$. Let $g_1,g_2\in{\mathcal S}(M)$ defined by
$$ 
g_1:=G_1\circ h\circ f_1\quad\&\quad g_2:=G_2\circ h\circ f_1,
$$
whose product $g_1g_2=0\in\gtp$ because $G_1G_2=0$. Since $\gtp$ is a prime ideal, either $g_1\in\gtp$ or $g_2\in\gtp$, and $\gtp\subset\gta_1\cap\gta_2$. Henceforth, either $g_1\in\gta_2$ or $g_2\in\gta_1$. We will get a contradiction by showing that both statements are false. To that end it suffices to check that $1-g_1\in\gta_2$ and $1-g_2\in\gta_1$. Thus  it is enough to see that
\begin{equation}\label{igualdad10}
1-g_1\in f_2{\mathcal S}(M)\quad\&\quad 1-g_2\in f_1{\mathcal S}(M).
\end{equation}
Let us check first that
\begin{equation}\label{igualdad20}
\Zz_M(f_2)\subset\Int_M(\Zz_M(1-g_1))\quad\&\quad\Zz_M(f_1)\subset\Int_M(\Zz_M(1-g_2)).
\end{equation}

\noindent To prove the first inclusion note that the open subset $U:=f_1^{-1}((1/2,+\infty))$ of $M$ contains $\Zz_M(f_2)$. Hence, it suffices to show that $U\subset\Zz_M(1-g_1)$. Indeed, for every $x\in U$ we have $H (f_1(x))=1$, thus $h (f_1(x))=1$, which implies $g_1(x)=G_1(1)=1$. 

Analogously, $V:=f_1^{-1}((-\infty, 1/3))$ is an open neighbourhood in $M$ of $\Zz_M(f_1)$, and it is enough to check that $V\subset\Zz_M(1-g_2)$. But $H(f_1(x))=0$ for every point $x\in V$, so $h(f_1(x))=-4$. Consequently $g_2(x)=G_2(-4)=G_1(0)=1$, which proves \eqref{igualdad20}. Therefore, by Lemma \ref{quot1}, there exist $u,v\in{\mathcal S}(M)$ such that 
\begin{multline*}
\Zz_M(1-g_1)\subset\Zz_M(u),\quad\Zz_M(1-g_2)\subset\Zz_M(v),\\ 
\quad1-g_1=uf_{2}\in f_{2}{\mathcal S}(M)\quad\&\quad1-g_2=vf_{1}\in f_1{\mathcal S}(M),
\end{multline*}
which proves equalities \eqref{igualdad10}.
\end{proof}

\vspace{3mm}

\begin{point}{{\bf Real spectrum of rings of semialgebraic functions}} \label{rspc} For the definition and main properties of the real spectrum of a ring see \cite[Ch.7]{bcr}. Let $\Spec_r({\mathcal S}(M))$ be the real spectrum of ${\mathcal S}(M)$, consisting of all \em prime cones \em in ${\mathcal S}(M)$, that is, those subsets $\alpha\subset{\mathcal S}(M)$ such that $f^2\in\alpha$ for every $f\in{\mathcal S}(M)$ and 
$$
i)\hspace{2mm}\alpha+\alpha\subset\alpha; \quad ii)\hspace{2mm}\alpha\cdot\alpha\subset\alpha; \quad iii)\hspace{1mm}-1\notin\alpha\quad\&\quad iv)\hspace{2mm}\text{ If } fg\in\alpha\text{ then either} f\in\alpha\text{ or }-g\in\alpha.
$$ 
It was proved in \cite{s} that ${\mathcal S}(M)$ is a real closed ring. This implies that the \em support map \em 
$$
\Spec_r({\mathcal S}(M))\to\Specs(M),\, \alpha\mapsto \gtp(\alpha):=\alpha\cap(-\alpha)
$$
is bijective and, obviously, it preserves inclusions. For every $\alpha\in\Spec_r({\mathcal S}(M))$ we denote $\kappa_{\alpha}$ the quotient ring of ${\mathcal S}(M)/\gtp(\alpha)$. It was firstly proved in \cite{s1} and later on in a different way in \cite{g}, that $\kappa_{\alpha}$ is a real closed field whose strictly positive elements are the quotients $(f+\gtp(\alpha))/(g+\gtp(\alpha))$ such that $fg\in\alpha\setminus(-\alpha)$. Given $f\in{\mathcal S}(M)$ it is usual to denote $f(\alpha):=f+\gtp(\alpha)\in\kappa_{\alpha}$. Hence, $f(\alpha)>0$ is equivalent to $f\in\alpha\setminus(-\alpha)$.
\end{point}

\vspace{3mm}

\begin{point}{{\bf Induced spectral morphisms}} \label{spcmpp} Let $\pi:M\to N$ be a semialgebraic map between semialgebraic subsets $M$ and $N$. All through this article we denote 
$$
\varphi_{\pi}:{\mathcal S}(N)\to{\mathcal S}(M),\, f\mapsto f\circ\pi.
$$  
This ring homomorphism induces a continuous map
$$
\Specs(\pi):\Specs(M)\to\Specs(N),\, \gtp\mapsto\varphi_{\pi}^{-1}(\gtp)
$$
where both spaces are endowed with the Zariski topology. Denote respectively $\betas M$ and $\betas N$ the subsets of $\Specs(M)$ and $\Specs(N)$ consisting of the maximal ideals of ${\mathcal S}(M)$ and ${\mathcal S}(N)$. Let ${\tt r}_N:\Specs(N)\to\betas N$ be the retraction that maps each prime ideal in ${\mathcal S}(N)$ to the unique maximal ideal containing it. It is proved in \cite[1.6.2]{b}, \cite[1.2]{dmo} that ${\tt r}_N$ is continuous. Thus, the composition 
$$
\betas\pi:={\tt r}_N\circ\Specs(\pi)|_{\betas M}:\betas M\to\betas N
$$ 
is continuous too. Given a function $f\in{\mathcal S}(M)$ we will denote
$$
\Dd_{\Specs(M)}(f):=\{\gtp\in\Specs(M):\, f\notin\gtp\}\quad\&\quad\Dd_{\betas M}(f):=\Dd_{\Specs(M)}(f)\cap\betas M.
$$
For each $x\in M$ consider the maximal ideal $\gtm_x:=\{f\in{\mathcal S}(M):\, f(x)=0\}$ of ${\mathcal S}(M)$, and let ${\tt j}_M:M\hookrightarrow\betas M,\, x\mapsto\gtm_x$. It is proved in \cite[4.7]{fg2} that the pair $(\betas M,{\tt j}_M)$ is a Hausdorff compactification of $M$, which is called the \em semialgebraic Stone--\v{C}ech compactification \em of $M$. Since ${\tt j}_N(N)$ and ${\tt j}_M(M)$ are respectively dense in $\betas N$ and $\betas M$, the map $\betas\pi$ is the unique continuous map making commutative the square 
$$
\xymatrix{
M\,\ar[d]_\pi\ar@{^{(}->}[r]^{\hspace{4mm}{\tt j}_M\quad}&\betas M\ar[d]_{\betas\pi}\\
N\,\ar@{^{(}->}[r]^{\hspace{4mm}{\tt j}_N\quad}&\betas N
}
$$
The difference $\partial M:=\betas M\setminus{\tt j}_M(M)$ is \em the remainder \em of the compactification $(\betas M,{\tt j}_M)$. 
\end{point}

To finish this preliminary section we recall some elementary notions concerning finiteness of ring homomorphisms and fix some notations. 

\begin{defns} \em (1) Given two rings (commutative with unity) $A$ and $B$, and a ring homomorphism $\varphi:A\to B$ we consider in $B$ the structure of $A$-algebra given by the multiplication $a\cdot b:=\varphi(a)b$ for every $a\in A$ and $b\in B$.

\noindent (2) It is said that $\varphi$ is \em finite, \em or that $B$ is a \em finite \em $A$-algebra, if there exist finitely many elements $b_1,\dots,b_n\in B$ such that $B=Ab_1+\cdots+Ab_n$.

\noindent (3) An element $b\in B$ is \em integral \em over $A$ if there exists a monic polynomial ${\tt p}({\tt t})\in A[\tt t]$ such that ${\tt p}(b)=0$. If every element in $B$ is integral over $A$ then it is said that $B$ is an \em integral \em $A$-algebra, or that $\varphi$ is an \em integral \em homomorphism.

\noindent (4) Given $b_1,\dots,b_n\in B$ the image of the evaluation homomorphism
$$
A[{\tt t}_1,\dots,{\tt t}_n]\to B,\, {\tt p}({\tt t}_1,\dots,{\tt t}_n)\mapsto {\tt p}(b_1,\dots,b_n)
$$ 
is denoted $A[b_1,\dots,b_n]$. If this homomorphism is surjective it is said that $B$ is a \em finitely generated \em $A$-algebra. If $n=1$ it is said that $\varphi$ is \em simple \em or that $B$ is a \em simple \em $A$-algebra.

Recall (see e.g.  \cite[5.2]{am}), that $B$ is a finite $A$-algebra if and only if $B$ is an integral and finitely generated $A$-algebra.

\noindent (5) For every ideal $\gta$ in $A$ we will denote $\gta B$ the smallest ideal of $B$ containing $\varphi(\gta)$,  without any explicit mention to the homomorphism $\varphi$.

\end{defns}

\begin{remark} \label{cons}\em Consider the following commutative square of ring homomorphisms with unity,   
$$
\hspace{2.5cm}\xymatrix{
A_1\ar[rr]_{\quad}\ar@{->}[d]&&B_1\ar@{->}[d]\\
A_2\ar[rr]_{\quad}&&B_2&&
}
$$
and suppose that the homomorphism $B_1\to B_2$ is surjective. Then, it follows straightforwardly from the definitions that if $B_1$ is a finite (resp. integral, simple or finitely generated) $A_1$-algebra then $B_2$ is a finite (resp. integral, simple or finitely generated) $A_2$-algebra.
\end{remark}

\section{Finiteness of semialgebraic homomorphisms}\label{s3}

\begin{prop} \label{fit} Let $M\subset\R^m$ be a semialgebraic set and let ${\tt j}:\R\hookrightarrow{\mathcal S}(M)$ be the ring homomorphism that maps each real number $r$ to the constant function 
$$
{\tt j}(r):M\to\R,\, x\mapsto r,
$$
which endows ${\mathcal S}(M)$ with an structure of $\R$-algebra. Then, the following conditions are equivalent:

\noindent \em (1) \em $M$ is a finite set.

\noindent \em (2) \em ${\mathcal S}(M)$ is a finite $\R$-algebra.

\noindent \em (3) \em ${\mathcal S}(M)$ is an integral $\R$-algebra.

\noindent \em (4) \em ${\mathcal S}(M)$ is a simple $\R$-algebra.

\noindent \em (5) \em ${\mathcal S}(M)$ is a finitely generated $\R$-algebra.

\noindent \em (6) \em ${\mathcal S}(M)$ is an artinian ring.
\end{prop}

\begin{proof} We will show that 
$$
(1)\Longrightarrow (2)\Longrightarrow (3)\Longrightarrow (1)\quad\&\quad (1)\Longrightarrow (4)\Longrightarrow (5)\Longrightarrow (6)\Longrightarrow (1).
$$ 
Suppose first that $M$ is finite with, say, $k$ elements. Then the topology in $M$ is discrete, and ${\mathcal S}(M)=\R^k$ is a finite $\R$-algebra. The implication $(2)\Longrightarrow (3)$ follows from \cite[Prop. 5.1]{am}, and suppose now that condition $(3)$ holds true. Let $f_i:=\pi_i|_M\in{\mathcal S}(M)$ be the restriction of the projection 
$$
\pi_i:\R^m\to\R,\,x:=(x_1,\dots,x_m)\to x_i. 
$$
For $1\leq i\leq m$ there exists a monic polynomial ${\tt p}_i\in\R[\t]$ such that ${\tt p}_i(f_i)=0$. Therefore $f_i(x)$ is, for every $x\in M$, one of the (finitely many) roots of ${\tt p}_i(\t)$, so $f_i(M)$ is finite. Consequently, $M\subset\prod_{i=1}^mf_i(M)$ is finite too.

Let us prove now that $(1)\Longrightarrow (4)$. The finiteness of $M:=\{p_1,\ldots,p_k\}\subset\R^m$ implies the existence of a function $f\in{\mathcal S}(M)$ such that $f(p_i)\neq f(p_j)$ if $i\neq j$. We prove this by induction on $m$, the case $m=1$ being trivial because the function $f:M\to\R,\, x\mapsto x$ do the job. 

For $m>1$ we can suppose, after a linear change of coordinates, that $p_i:=(a_{1i},\ldots,a_{mi})$ with $a_{1i}\neq a_{1j}$ if $i\neq j$. Let 
$$
\pi:\R^m\to\R^{m-1},\, x:=(x_1,\ldots,x_m)\mapsto x':=(x_1,\ldots,x_{m-1})
$$
be the projection onto the first $m-1$ coordinates and notice that the points $q_j:=\pi(p_j)$ are pairwise different for $1\leq j\leq k$. Denote $N:=\{q_1,\dots,q_k\}\subset\R^{m-1}$ and, by the inductive hypothesis, there exists $g\in{\mathcal S}(N)$ such that $g(q_i)\neq g(q_j)$ if $i\neq j$. As $\pi(M)=N$ the function $f:=g\circ\pi|_M\in{\mathcal S}(M)$ satisfies $f(p_i)=g(q_i)\neq g(q_j)=f(p_j)$ if $i\neq j$.

Let us check now that ${\mathcal S}(M)=\R[f]$, which proves $(4)$. Indeed, let $a_j:=f(p_j)$ for $1\leq j\leq k$ and observe that each function
$$
f_i:=\frac{\prod_{j\neq i}(f-a_j)}{\prod_{j\neq i}(a_i-a_j)}\in\R[f]
$$
satisfies $f_i(p_i)=1$ and $f_i(p_j)=0$ if $i\neq j$. Thus, each function $g\in{\mathcal S}(M)$ satisfies 
$$
g=\sum_{i=1}^kg(p_i)f_i\in\R[f],
$$ 
as wanted.

The implication $(4)\Longrightarrow (5)$ is evident. Suppose now that ${\mathcal S}(M)$ is a finitely generated $\R$-algebra. In particular it is a noetherian ring, so it has finitely many minimal prime ideals, because the radical $\sqrt{(0)}$ of the zero ideal is an intersection of finitely many prime ideals and the minimal prime ideals of ${\mathcal S}(M)$ are among them. On the other hand  every prime ideal in ${\mathcal S}(M)$ is contained, by Theorem \ref{Gelfand}, in a unique maximal ideal. Thus, there exists a well defined map 
$$
{\tt r}_M:\Min\,({\mathcal S}(M))\to\Max\,({\mathcal S}(M)),\, \gtp\mapsto{\tt r}(\gtp)
$$
from the set $\Min\,({\mathcal S}(M))$ of minimal prime ideals of ${\mathcal S}(M)$ to the set $\Max\,({\mathcal S}(M))$ of maximal ideals, where ${\tt r}_M(\gtp)$ is the unique maximal ideal of ${\mathcal S}(M)$ containing $\gtp$. In addition, the map ${\tt r}_M$ is surjective because the Krull dimension of ${\mathcal S}(M)$ is finite; indeed it coincides with $\dim\, (M)$, see e.g. \cite[Prop. 1.4]{gr} or \cite[Thm. 1.1]{fg1}. Hence, also $\Max\,({\mathcal S}(M))$ is a finite set. As the map $M\to\Max({\mathcal S}(M)),\, x\mapsto\gtm_x$ is injective, $M$ is finite too.  Consequently, the Krull dimension of ${\mathcal S}(M)$ equals $\dim\,(M)=0$. Therefore ${\mathcal S}(M)$ is a zero dimensional noetherian ring and, by \cite[Thm 8.5]{am}, it is an artinian ring.

Finally, if ${\mathcal S}(M)$ is artinian then $\dim(M)=\dim({\mathcal S}(M))=0$, and $M$ is finite.
\end{proof}

\begin{cor} \label{ff} Let $\pi:M\to N$ be a semialgebraic map and suppose that the induced homomorphism $\varphi_{\pi}:{\mathcal S}(N)\to{\mathcal S}(M)$ is finite \em (\em resp. integral, finitely generated, simple\em)\em. Then, the fibers of $\pi$ are finite sets.
\end{cor}
\begin{proof} Pick a point $p\in N$ and consider the fiber $P:=\pi^{-1}(p)$, which is a closed semialgebraic subset of $M$. By Lemma \ref{tietze} the homomorphism ${\mathcal S}(M)\to{\mathcal S}(P),\, f\mapsto f|_P$ is surjective. Consider the commutative square 
$$
\hspace{2cm}\xymatrix{
{\mathcal S}(N)\ar[rr]^{\varphi_{\pi}\quad}_{\quad}\ar@{->}[d]&&{\mathcal S}(M)\ar@{->}[d]\\
{\mathcal S}(\{p\})=\R\ar[rr]^{\tt j\quad}_{\quad}&&{\mathcal S}(P)&&
}
$$
where the vertical arrows are the restriction mappings and ${\tt j}$ transforms each real number $r$ into the constant function that only attains the value $r$. Applying Remark \ref{cons} it follows that ${\mathcal S}(P)$ is a finite $\R$-algebra and, by Proposition \ref{fit}, the fiber $P$ is finite.
\end{proof}

\begin{cor} \label{nfin} Let $\pi:M\to N$ be a surjective semialgebraic map between two semialgebraic sets $M$ and $N$, where $N$ is finite, and let $\varphi_{\pi}:{\mathcal S}(N)\to{\mathcal S}(M)$ be the induced homomorphism.  Then, the following conditions are equivalent:

\vspace{1mm}

\noindent \em (1) \em The set $M$ is finite.

\vspace{1mm}

\noindent \em (2) \em The homomorphism $\varphi_{\pi}$ is finite. 

\vspace{1mm}

\noindent \em (3) \em The homomorphism $\varphi_{\pi}$ is integral.

\vspace{1mm} 

\noindent \em (4) \em The homomorphism $\varphi_{\pi}$ is simple. 

\vspace{1mm}

\noindent \em (5) \em The homomorphism $\varphi_{\pi}$ is finitely generated. 
\end{cor}

\begin{proof} Notice that $\R\hookrightarrow{\mathcal S}(N)\to{\mathcal S}(M)$. Thus, if $M$ is finite it follows from Proposition \ref{fit} that ${\mathcal S}(M)$ is a finite, integral, simple and finitely generated $\R$-algebra, so it is a finite, integral, simple and finitely generated ${\mathcal S}(N)$-algebra. This proves that condition (1) implies conditions (2), (3), (4) and (5). 

Conversely, suppose that one among conditions (2), (3), (4) and (5) holds. Then, by Corollary \ref{ff}, the fibers of $\pi$ are finite. Hence, $N$ being finite, the same holds for $M$.
\end{proof}

\begin{prop} \label{ci} Let $\pi:M\to N$ be a closed and surjective semialgebraic map between two semialgebraic sets $M$ and $N$. Let $Y$ be a closed semialgebraic subset of $N$ and let us denote $X:=\pi^{-1}(Y)$. Suppose that the map 
$$
\pi|_{M\setminus X}:M\setminus X\to N\setminus Y
$$ 
is injective. Then, the homomorphism $\varphi_{\pi}:{\mathcal S}(N)\to{\mathcal S}(M)$ is finite \em (\em resp. integral, finitely generated or simple\em) \em if and only if ${\widetilde\varphi}_{\pi}:{\mathcal S}(Y)\to{\mathcal S}(X),\, g\mapsto g\circ(\pi|_X)$ is finite \em (\em integral, finitely generated or simple\em).
\end{prop}

\begin{proof} The only if part follows at once from Remark \ref{cons} because we have a commutative square
$$
\hspace{2.5cm}\xymatrix{
{\mathcal S}(N)\ar[rr]^{\hspace{4mm}\varphi_{\pi}\quad}_{\quad}\ar@{->}[d]&&{\mathcal S}(M)\ar@{->}[d]\\
{\mathcal S}(Y)\ar[rr]^{\hspace{4mm}\widetilde\varphi_{\pi}\quad}_{\quad}&&{\mathcal S}(X)&&
}
$$
where the vertical arrows are the restriction homomorphisms, which are surjective by Lemma \ref{tietze} because $X$ and $Y$ are, respectively, closed subsets of $M$ and $N$. For the converse, denote 
$$
{\J}(X):=\{f\in{\mathcal S}(M):X\subset\Zz_M(f)\}\quad\&\quad{\J}(Y):=\{g\in{\mathcal S}(N):  Y\subset\Zz_N(g)\}.
$$ 
Let us see first that $\varphi_{\pi}({\J}(Y))={\J}(X)$. For $g\in{\J}(Y)$ and $x\in X$ we have $\pi(x)\in Y$ and $(\varphi_{\pi}(g))(x)=(g\circ\pi)(x)=g(\pi(x))=0$, which proves the inclusion $\varphi_{\pi}({\J}(Y))\subset{\J}(X)$. Conversely, let $f\in{\J}(X)$ and define the function 
$$
g:N\to\R,\, v\mapsto f(u)\ \text{ where }\ u\in\pi^{-1}(v)\ \text{ is arbitrary }.
$$
Let us prove that $g\in{\mathcal S}(N)$. To check that $g$ is a well defined function it suffices to see that for every point $v\in N$ the function $f$ is constant on the (nonempty) fiber $\pi^{-1}(v)$. In case $v\in N\setminus Y$ there is nothing to check  because $\pi^{-1}(v)$ is a singleton. On the other hand, if $v\in Y$ and $u\in\pi^{-1}(v)\subset\pi^{-1}(Y)=X$ we have $f(u)=0$. To prove that $g$ is continuous note that the triangle
$$
\xymatrix{
M\ar@{->}[r]^{\quad\pi\quad}\ar@{->}[rd]^{f\quad}&N\ar@{->}[d]^{g\quad}\\
&\R}
$$
is commutative.  Thus, the preimage $g^{-1}(C)=\pi(f^{-1}(C))$ of a closed semialgebraic subset $C$ of $\R$ is a closed subset of $N$ because $f$ is continuous and $\pi$ is a closed semialgebraic map.  To see that $g$ is a semialgebraic function, and therefore $g\in{\mathcal S}(N)$, it suffices to check that its graph $\Gamma(g)$ is a semialgebraic subset of $N\times\R$. Indeed, the mapping
$$
\rho:M\to N\times\R,\, x\mapsto (\pi(x),f(x))
$$
is semialgebraic and, since $\pi:M\to N$ is surjective, $\Gamma(g)=\rho(M)$ is a semialgebraic set.

Thus $g\in{\mathcal S}(N)$ and, in fact, $g\in{\J}(Y)$ because given $y\in Y$ there exists, since $\pi$ is surjective, a point $x\in X$ such that $y=\pi(x)$, so $g(y)=g(\pi(x))=f(x)=0$. Moreover, $\varphi_{\pi}(g)=g\circ\pi=f$ and the equality $\varphi_{\pi}({\J}(Y))={\J}(X)$ is proved.

Suppose now that ${\widetilde\varphi}_{\pi}$ is a finite homomorphism and let us prove that $\varphi_{\pi}$ is finite too. Let $g_1,\dots,g_r\in{\mathcal S}(X)$ such that 
$$
{\mathcal S}(X)={\mathcal S}(Y)g_1+\cdots+{\mathcal S}(Y)g_r,
$$
The constant function $G_{r+1}:M\to\R,\, x\mapsto 1$ belongs to ${\mathcal S}(M)$ and, by Lemma \ref{tietze} (2), there exist $G_1,\dots,G_r\in{\mathcal S}(M)$ with $G_i|_X=g_i$ for $1\leq i\leq r$. Let us check the equality
$$
{\mathcal S}(M)={\mathcal S}(N)G_1+\cdots+{\mathcal S}(N)G_r+{\mathcal S}(N)G_{r+1}.
$$
Given $F\in{\mathcal S}(M)$ its restriction $f:=F|_X\in{\mathcal S}(X)$. Thus, there exist $h_1,\dots,h_r\in{\mathcal S}(Y)$ such that
\begin{equation}\label{EQQ}
f={\widetilde\varphi}_{\pi}(h_1)g_1+\cdots+{\widetilde\varphi}_{\pi}(h_r)g_r=(h_1\circ\pi|_X)g_1+\cdots+(h_r\circ\pi|_X)g_r. 
\end{equation}
For $1\leq i\leq r$ let $H_i\in{\mathcal S}(N)$ with $H_i|_Y=h_i$. By \eqref{EQQ} we have
$$
F-\sum_{i=1}^r(H_i\circ\pi)G_i\in{\J}(X). 
$$
Since we have proved that $\varphi_{\pi}({\J}(Y))={\J}(X)$ there exists a function $H_{r+1}\in{\J}(Y)$ with
$$
F-\sum_{i=1}^r(H_i\circ\pi)G_i=\varphi_{\pi}(H_{r+1})=H_{r+1}\circ\pi=(H_{r+1}\circ\pi)G_{r+1}.
$$ 
In other words, 
$$
F=\sum_{i=1}^{r+1}(H_i\circ\pi)G_i=\sum_{i=1}^{r+1}\varphi_{\pi}(H_i)G_i\in{\mathcal S}(N)G_1+\cdots+{\mathcal S}(N)G_r+{\mathcal S}(N)G_{r+1}.
$$
Next, suppose that ${\widetilde\varphi}_{\pi}$ is a simple homomorphism and let us prove that $\varphi_{\pi}$ enjoys this property too. Let $f\in{\mathcal S}(X)$ such that ${\mathcal S}(X)={\mathcal S}(Y)[f]$ and let $F\in{\mathcal S}(M)$ with $F|_X=f$. We will see that ${\mathcal S}(M)={\mathcal S}(N)[F]$. For each $H\in{\mathcal S}(M)$ its restriction $h:=H|_X\in{\mathcal S}(X)$. Hence, there exists a polynomial ${\tt p}\in{\mathcal S}(Y)[\tt t]$ such that $h={\tt p}(f)$. Write
$$
{\tt p}({\tt t}):=\sum_{i=0}^ng_i{\tt t}^{n-i},\ \text{ where }\ g_i\in{\mathcal S}(Y)
$$
and, for $0\leq i\leq n$, let $G_i\in{\mathcal S}(N)$ with $G_i|_{Y}=g_i$. As $\sum_{i=0}^n{\widetilde\varphi}_{\pi}(g_i)f^{n-i}={\tt p}(f)=h$ we have  
$$
H-\sum_{i=0}^n\varphi_{\pi}(G_i)F^{n-i}=H-\sum_{i=0}^n(G_i\circ\pi)F^{n-i}\in{\J}(X)=\varphi_{\pi}({\J}(Y)).
$$
Therefore there exists $G\in{\J}(Y)\subset{\mathcal S}(N)$ such that
$$
H-\sum_{i=0}^n\varphi_{\pi}(G_i)F^{n-i}=\varphi_{\pi}(G).
$$
Consequently, the polynomial ${\tt P}(\t):=(G_n+G)+\sum_{i=0}^{n-1}G_i\t^{n-i}\in{\mathcal S}(N)[\tt t]$ and $H={\tt P}(F)$.

Arguing analogously one proves that if ${\widetilde\varphi}_{\pi}$ is a finitely generated homomorphism the same holds true for $\varphi_{\pi}$. So to finish it suffices to see that $\varphi_{\pi}$ is an integral homomorphism whenever ${\widetilde\varphi}_{\pi}$ is so. Given $F\in{\mathcal S}(M)$ its restriction $f:=F|_X$ is integral over ${\mathcal S}(Y)$, that is, there exist $g_1,\dots,g_n\in{\mathcal S}(Y)$ such that
$$
f^n+\sum_{i=1}^n{\widetilde\varphi}_{\pi}(g_i)f^{n-i}=0.
$$
For $1\leq i\leq n$ let $G_i\in{\mathcal S}(N)$ such that $G_i|_Y=g_i$. Then $F^n+\sum_{i=1}^n\varphi_{\pi}(G_i)F^{n-i}\in{\J}(X)$. Since ${\J}(X)=\varphi_{\pi}({\J}(Y))$ there exists $G\in{\J}(Y)\subset{\mathcal S}(N)$ such that
$$
F^n+\sum_{i=1}^n\varphi_{\pi}(G_i)F^{n-i}=\varphi_{\pi}(G).
$$
Thus ${\tt P}(F)=0$, where
$$
{\tt P}({\tt t}):={\t}^n+\sum_{i=1}^{n-1}G_i{\tt t}^{n-i}+(G_n-G)\in{\mathcal S}(N)[\t]
$$
is a monic polynomial, so $F$ is integral over ${\mathcal S}(N)$.
\end{proof}

\begin{cor} \label{mfin} Let $\pi:M\to N$ be a surjective semialgebraic map between two semialgebraic sets $M$ and $N$ such that $Y:=\{y\in N:\cd(\pi^{-1}(y))>1\}$ is a finite set. Then, the following statements are equivalent:

\vspace{1mm}

\noindent \em (1) \em The set $X:=\pi^{-1}(Y)$ is finite. 

\vspace{1mm}

\noindent \em (2) \em The homomorphism $\varphi_{\pi}$ is finite. 

\vspace{1mm}

\noindent \em (3) \em The homomorphism $\varphi_{\pi}$ is integral. 

\vspace{1mm}

\noindent \em (4) \em The homomorphism $\varphi_{\pi}$ is simple. 

\vspace{1mm}

\noindent \em (5) \em The homomorphism $\varphi_{\pi}$ is finitely generated, that is, ${\mathcal S}(M)$ is a finitely generated ${\mathcal S}(N)$-algebra.
\end{cor}

\begin{proof} Notice that $Y$ being finite it is a closed semialgebraic subset of $N$ and, by its definition, the restriction
$$
\pi|_{M\setminus X}:M\setminus X\to N\setminus Y
$$
is injective. It follows from Proposition \ref{ci} that the ring homomorphism $\varphi_{\pi}$ is finite (resp. integral, simple or finitely generated) if and only if the ring homomorphism 
$$
{\widetilde\varphi}_{\pi}:{\mathcal S}(Y)\to{\mathcal S}(X),\, g\mapsto g\circ(\pi|_X)
$$ 
is finite (resp. integral, simple or finitely generated). Each one of these conditions is equivalent, by Corollary \ref{nfin}, to the finiteness of $X$, as wanted.
\end{proof}

Next we apply the results above to semialgebraic compactifications of locally compact semialgebraic sets.

\begin{defns} \em (1) A \em semialgebraic compactification \em of a semialgebraic set $M$ is a pair $(X,{\tt j})$ where $X$ is a compact semialgebraic set and ${\tt j}:M\to X$ is a semialgebraic homeomorphism onto its image ${\tt j}(M)$, which is a dense subset of $X$. We denote $\partial X:=X\setminus {\tt j}(M)$ the \em remainder \em of the compactification $(X,{\tt j})$ of $M$.

\vspace{1mm}

\noindent (2) Given two semialgebraic compactifications $(X_1,{\tt j}_1)$ and $(X_2,{\tt j}_2)$ of a semialgebraic set $M$, we say that $(X_2,{\tt j}_2)$ \em dominates \em $(X_1,{\tt j}_1)$, and we write $(X_1,{\tt j}_1)\preccurlyeq(X_2,{\tt j}_2)$, if there exists a semialgebraic map $\pi:X_2\to X_1$ such that $\pi\circ {\tt j}_2={\tt j}_1$. Notice that $\pi$ is surjective since its image is a compact subset of $X_1$ that contains its dense subset $\pi({\tt j}_2(M))={\tt j}_1(M)$. Moreover, the equality $\pi\circ{\tt j}_2={\tt j}_1$ determines $\pi$ because for every point $p\in X_2$ there exists a sequence $\{x_n\}\subset M$ such that $p=\lim_{n\to\infty}\{{\tt j}_2(x_n)\}$, and  
$$
\pi(p)=\lim_{n\to\infty}\{\pi({\tt j}_2(x_n))\}=\lim_{n\to\infty}\{{\tt j}_1(x_n)\}.
$$
Then ${\mathcal S}(X_2)$ is an ${\mathcal S}(X_1)$-algebra with the structure endowed by the homomorphism
$$
\varphi_{\pi}:{\mathcal S}(X_1)\to{\mathcal S}(X_2),\, f\mapsto f\circ\pi.
$$
As $\pi^{-1}(\partial X_1)=\partial X_2$ by \cite[4.3]{fg2}, $\pi|_{\partial X_2}:\partial X_2\to\partial X_1$ is a surjective semialgebraic mapping. Therefore, ${\mathcal S}(\partial X_2)$ is a ${\mathcal S}(\partial X_1)$-algebra with the structure endowed by the homomorphism
$$
\partial\varphi_{\pi}:{\mathcal S}(\partial X_1)\to{\mathcal S}(\partial X_2),\, f\mapsto f\circ(\pi|_{\partial X_2}).
$$
\end{defns}

\begin{cor} \label{borde} Let $(X_1,{\tt j}_1)$ and $(X_2,{\tt j}_2)$ be two semialgebraic compactifications of a non compact semialgebraic set $M$ such that $(X_2,{\tt j}_2)$ dominates  $(X_1,{\tt j}_1)$ and $\partial X_1$ is a closed subset of $X_1$. Then, ${\mathcal S}(X_2)$ is a finite \em (\em resp. integral, simple or finitely generated\,\em) \em ${\mathcal S}(X_1)$-algebra if and only if ${\mathcal S}(\partial X_2)$ is a finite \em (\em resp. integral, simple or finitely generated\,\em) ${\mathcal S}(\partial X_1)$-\em algebra. 
\end{cor}

\begin{proof} Let $\pi:X_2\to X_1$ be the unique semialgebraic map satisfying $\pi\circ{\tt j}_2={\tt j}_1$. As we have just remarked, $\pi^{-1}(\partial X_1)=\partial X_2$. In addition, the restriction
$$
\pi|_{X_2\setminus\pi^{-1}(\partial X_1)}:X_2\setminus\pi^{-1}(\partial X_1)={\tt j}_2(M)\to X_1\setminus\partial X_1={\tt j}_1(M)
$$
is injective because $\pi|_{X_2\setminus\pi^{-1}(\partial X_1)}={\tt j}_1\circ{\tt j}_2^{-1}$ and both ${\tt j}_1$ and ${\tt j}_2$ are injective mappings. Hence the result follows from Proposition \ref{ci}
\end{proof}

\begin{remark} \label{extra0}\em If $M$ is locally compact then the residue $\partial X_1$ is, by \cite[4.14]{fg2}, a closed semialgebraic subset of $X_1$. Thus, the last Corollary \ref{borde} holds true in this case.
\end{remark}

\begin{cor} Let $(X_1,{\tt j}_1)$ and $(X_2,{\tt j}_2)$ be two semialgebraic compactifications of a semialgebraic set $M$ such that $(X_2,{\tt j}_2)$ dominates $(X_1,{\tt j}_1)$. 

\noindent \em (1) \em If ${\mathcal S}(X_2)$ is a finite \em (\em resp. simple, finitely generated\,\em) \em ${\mathcal S}(X_1)$-algebra, then the fibers of the unique semialgebraic map $\pi:X_2\to X_1$ such that $\pi\circ{\tt j}_2={\tt j}_1$ are finite.

\noindent \em (2) \em Suppose that the union of the fibers of $\pi$ with more than a point is a finite set. Then ${\mathcal S}(X_2)$ is a finite \em (\em resp. simple, finitely generated\,\em) \em ${\mathcal S}(X_1)$-\em algebra.
\end{cor}
\begin{proof} (1) This is a particular case of Corollary \ref{ff}.

\noindent (2) Denote $Y:=\{y\in X_1:\cd(\pi^{-1}(y))>1\}$. By hypothesis $\pi^{-1}(Y)$ is finite and, since $\pi$ is surjective, $Y$ is finite too. Then, the conclusion follows at once from Corollary \ref{mfin}. 
\end{proof}

\begin{cor} Let $(X_1,{\tt j}_1)$ and $(X_2,{\tt j}_2)$ be two semialgebraic compactifications of a semialgebraic set $M$ such that $(X_2,{\tt j}_2)$ dominates $(X_1,{\tt j}_1)$ and the remainder $\partial X_1$ is a finite set. Then, the following conditions are equivalent:

\noindent\em (1) \em The remainder $\partial X_2$ is finite.

\vspace{1mm}

\noindent\em (2) \em ${\mathcal S}(X_2)$ is a finite ${\mathcal S}(X_1)$-algebra.

\vspace{1mm}

\noindent\em (3) \em ${\mathcal S}(X_2)$ is integral over ${\mathcal S}(X_1)$.

\vspace{1mm}

\noindent\em (4) \em ${\mathcal S}(X_2)$ is a finitely generated ${\mathcal S}(X_1)$-algebra.

\vspace{1mm}

\noindent \em (5) \em ${\mathcal S}(X_2)$ is a simple ${\mathcal S}(X_1)$-algebra.
\end{cor}
\begin{proof} Let $\pi:X_2\to X_1$ be the unique semialgebraic map such that $\pi\circ{\tt j}_2={\tt j}_1$. By \cite[4.3]{fg2}, $\pi^{-1}(\partial X_1)=\partial X_2$, so $\pi|_{\partial X_2}:\partial X_2\to\partial X_1$ is a surjective semialgebraic mapping. Since $\partial X_1$ is finite, and ${\mathcal S}(\partial X_2)$ is a ${\mathcal S}(\partial X_1)$-algebra with the structure endowed by the homomorphism
$$
\partial\varphi_{\pi}:{\mathcal S}(\partial X_1)\to{\mathcal S}(\partial X_2),\, f\mapsto f\circ(\pi|_{\partial X_2}),
$$
the finiteness of $\partial X_2$ is equivalent, by Corollary \ref {nfin} to any, and so all, the following conditions:

\vspace{1mm}

\noindent i) The homomorphism $\partial\varphi_{\pi}$ is finite; ii) $\partial\varphi_{\pi}$ is integral; iii) $\partial\varphi_{\pi}$ is simple and iv) $\partial\varphi_{\pi}$ is finitely generated. To finish it suffices to apply Corollary \ref{borde} because $\partial X_1$ being finite it is a closed subset of $X_1$.
\end{proof}

\section{Locally injective semialgebraic mappings}\label{s4}

\begin{prop} \label{integral} Let $\pi:M\to N$ be a semialgebraic map between semialgebraic subsets $M$ and $N$ such that the ring homomorphism $\varphi_{\pi}:{\mathcal S}(N)\to{\mathcal S}(M$ is integral. Then 

\noindent \em (1) \em The spectral map
$$
\Specs(\pi):\Specs(M)\to\Specs(N)
$$
is proper and separated \em(\em that is, its fibers are Hausdorff spaces\em)\em. 

\noindent \em (2) \em $\Specs(M)$ maps maximal ideals of ${\mathcal S}(M)$ into maximal ideals of ${\mathcal S}(N)$. Thus the restriction 
$$
\betas\pi:=\Specs(\pi)|_{\betas M}:\betas M\to\betas N
$$
is a well defined proper map and it is unique making commutative the square
\begin{equation*}
\xymatrix{
M\,\ar[d]_\pi\ar@{^{(}->}[r]^{{\hspace{3mm}\tt j}_M\quad}&\betas M\ar[d]_{\betas\pi}\\
N\,\ar@{^{(}->}[r]^{{\hspace{3mm}\tt j}_N\quad}&\betas N
}
\end{equation*}
\em (3) \em The map $\pi$ is proper with finite fibers. 
\end{prop}
\begin{proof} (1) Since $\varphi_{\pi}$ is integral the spectral map $\Specs(\pi)$ is closed, by \cite[Thm. 5.10]{am}. In addition, the fibers of $\Specs(\pi)$ are compact because they are closed subsets of the compact space $\Specs(M)$. Thus $\Specs(\pi)$ is a proper map. %This is well known but we supply a proof for the sake of completeness. Let $Z$ be a closed subset of $\Specs(M)$. Then, there exists an ideal $\gta$ in ${\mathcal S}(M)$ such that
%$$
%Z:=\{\gtp\in\Specs(M): \gta\subset\gtp\}.
%$$
%Let $\gtb:=\varphi_{\pi}^{-1}(\gta)$ and denote $Z':=\{\gtq\in\Specs(N): \gtb\subset\gtq\}$, which is a closed subset of $\Specs(N)$. All reduces to check that $\Specs(\pi)(Z)=Z'$. As each $\gtp_0\in Z$ contains $\gta$, the ideal $\Specs(\pi)(\gtp_0)=\varphi_{\pi}^{-1}(\gtp_0)$ contains $\gtb$, i.e. $\Specs(\pi)(\gtp_0)\in Z'$. This proves the inclusion $\Specs(\pi)(Z)\subset Z'$.

%For the converse note first that, since $\varphi_{\pi}$ is an integral homomorphism, the same holds for the induced homomorphism $\ol{\varphi}_{\pi}:{\mathcal S}(N)/\gtb\to{\mathcal S}(M)/\gta$. By \cite[5.10]{am}, the spectral map 
%$$
%\Spec(\ol{\varphi}_{\pi}):\Spec({\mathcal S}(M)/\gta)\to\Spec({\mathcal S}(N)/\gtb)
%$$
%is surjective. Hence, given $\gtq_0\in Z'$ the quotient $\gtq_0/\gtb$ is a prime ideal in ${\mathcal S}(N)/\gtb$ and there exists a prime ideal $\gtp_0$ in ${\mathcal S}(M)$ containing $\gta$ with $\Spec(\ol{\varphi}_{\pi})(\gtp_0/\gta)=\gtq_0/\gtb$. Therefore $\gtp_0\in Z$ and $\Specs(\pi)(\gtp_0)=\gtq_0$.

Let us see now that for every prime ideal $\gtq$ in ${\mathcal S}(N)$ the fiber $X:=\Specs(\pi)^{-1}(\gtq)$ is a Hausdorff space. Let $\gtp_1,\gtp_2$ be two distinct points in $X$. By \cite[5.9]{am}, $\gtp_1\not\subset\gtp_2$ and $\gtp_2\not\subset\gtp_1$.  Let $\alpha_1, \alpha_2\in\Spec_r({\mathcal S}(M))$ such that $\gtp(\alpha_i)=\gtp_i$ for $i=1,2$, and let $f_1\in\alpha_1\setminus\alpha_2$ and $f_2\in\alpha_2\setminus\alpha_1$. Thus 
$$
f_1(\alpha_1)\geq0,\quad f_1(\alpha_2)<0,\quad f_2(\alpha_2)\geq0\quad\&\quad f_2(\alpha_1)<0,
$$
which implies that $h:=f_1-f_2$ satisfies $h(\alpha_1)>0$ and $h(\alpha_2)<0$. The functions $f:=h+|h|$ and $g:=h-|h|$ satisfy $f(\alpha_1)>0$ and $g(\alpha_2)<0$, so $f\in{\mathcal S}(M)\setminus\gtp_1$, $g\in{\mathcal S}(M)\setminus\gtp_2$ and $fg=h^2-|h|^2=0$. Consequently, the open disjoints subsets
\begin{equation*}
\begin{split}
\Dd_{\Specs(M)}(f)&=\{\gtp\in\Specs(M):\, f\notin\gtp\}\quad\&\\
&\quad\Dd_{\Specs(M)}(g)=\{\gtp\in\Specs(M):\, g\notin\gtp\}
\end{split}
\end{equation*}
are, respectively, open neighbourhoods of $\gtp_1$ and $\gtp_2$ in $\Specs(M)$.

\noindent (2)  Let $(\betas M,{\tt j}_M)$ and $(\betas N,{\tt j}_N)$ be, respectively, the semialgebraic Stone--\v{C}ech compactifications of $M$ and $N$. The closed map $\Specs(\pi)$ transforms the closed points of $\Specs(M)$ into closed points of $\Specs(N)$. Thus, if ${\tt r}_N:\Specs(N)\to\betas N$ denotes the retraction that maps each prime ideal in ${\mathcal S}(N)$ to the unique maximal ideal containing it, we have
$$
\betas\pi={\tt r}_N\circ\Specs(\pi)|_{\betas M}=\Specs(\pi)|_{\betas M}.
$$ 
By \cite[4.4]{fg2} $\betas M$ and $\betas N$ are compact and Hausdorff spaces, so $\betas\pi$ is a proper map. Its uniqueness making commutative the square in the statement follows at once from (\ref{spcmpp}).

\noindent (3)  Let us prove that $(\betas\pi)^{-1}({\tt j}_N(N))={\tt j}_M(M)$. The inclusion ${\tt j}_M(M)\subset(\betas\pi)^{-1}({\tt j}_N(N))$ is an straightforward consequence of the commutativity of the square above. Conversely, given $\gtm\in(\betas\pi)^{-1}({\tt j}_N(N))$ the maximal ideal $\gtn:=\betas\pi(\gtm)\in{\tt j}_N(N)$, i.e., there exists a point $y\in N$ such that $\gtn=\gtm_y$. Thus ${\mathcal S}(N)/\gtn=\R$ and, since $\gtn=\varphi^{-1}(\gtm)$, we have an  algebraic field extension
$$
\R={\mathcal S}(N)/\gtn\hookrightarrow{\mathcal S}(M)/\gtm,
$$
because $\varphi_{\pi}$ is an integral homomorphism. By \cite[\S III.1]{s1} or \cite[Thm. 5.12]{s2} ${\mathcal S}(M)$ is a real closed ring, so the quotient ${\mathcal S}(M)/\gtm$ is a real closed field and the equality ${\mathcal S}(M)/\gtm=\R$ holds. This implies, by \cite[3.11]{fg2}, that there exists a point $x\in M$ such that $\gtm=\gtm_x={\tt j}_M(x)\in{\tt j}_M(M)$. Therefore,
$$
\betas\pi|_{{\tt j}_M(M)}:{\tt j}_M(M)\to{\tt j}_N(N)
$$
is a proper map too, and the same holds for
$$
\pi={\tt j}_N^{-1}\circ\betas\pi|_{{\tt j}_M(M)}\circ{\tt j}_M:M\to N.
$$
\noindent Finally, the finiteness of the fibers of $\pi$ follows at once from Corollary \ref{ff}.
\end{proof}

\begin{remark}\label{hardt} \em For every semialgebraic map $\pi:M\to N$ whose fibers are finite, as in Proposition \ref{integral}, there exists an integer $k>0$ such that the fibers of $\pi$ have, at most, $k$ elements. Indeed, by Hardt's trivialization theorem, see \cite[9.3.2]{bcr}, there exist finitely many subsets $T_1,\dots,T_{r}$ of $N$, semialgebraic sets $P_1,\dots, P_r$ and semialgebraic homeomorphisms $\theta_{\ell}:T_{\ell}\times P_{\ell}\to \pi^{-1}(T_{\ell})$
such that for $1\leq\ell\leq r$ we have
$$
\xymatrix{
T_{\ell}\times P_{\ell}\ar@{->}[r]^{\quad\theta_{\ell}\quad}\ar@{->}[rd]^{\pi_1\quad}&\pi^{-1}(T_{\ell})\ar@{->}[d]^{\pi\quad}\\
&T_{\ell}}
$$
where $\pi_1:T_{\ell}\times P_{\ell}\to T_{\ell}$ is the first projection. Hence, each set $P_{\ell}$ is finite and 
$$
\cd\,(\pi^{-1}(y))=\cd\,(P_{\ell})
$$ 
for every $y\in T_{\ell}$ and $1\leq\ell\leq r$. Thus, $k:=\max\,\{\cd\,(P_{\ell}):\, 1\leq \ell\leq r\}$ does the job.
\end{remark}

The next examples show that local injectivity is somehow related to integral extensions.

\begin{examples} \label{lcinycv} \em (1) Let $\pi:M\to N$ be a closed semialgebraic map between semialgebraic subsets $M$ and $N$. Suppose that there exists a finite cover $\{M_i:\, 1\leq i\leq k\}$ of $M$ by closed semialgebraic subsets such that each restriction $\pi|_{M_i}:M_i\to N$ is injective. Then the induced homomorphism $
\varphi_{\pi}:{\mathcal S}(N)\to{\mathcal S}(M)$ is integral. 

Notice that $N_i:=\pi(M_i)$ is a closed semialgebraic subset of $N$, because $\pi$ is a closed semialgebraic map, and $\pi|_{M_i}:M_i\to N_i$ is a closed and continuous semialgebraic bijection. Thu it is a semialgebraic homeomorphism. Consequently, for every $f\in{\mathcal S}(M)$ the function $f_i:=f|_{M_i}\circ(\pi|_{N_i})^{-1}:N_i\to\R$ is semialgebraic and, since $N_i$ is a closed semialgebraic subset of $N$, there exists, by Lemma \ref{tietze} (2), a semialgebraic extension $g_i:N\to\R$ of $f_i$. Now, for each point $x\in M$ there exists an index $1\leq i\leq k$ such that $x\in M_i$; thus $g_i(\pi(x))=f_i(\pi(x))=f(x)$, and henceforth the equality
$$
\prod_{i=1}^k(f-g_i\circ\pi)=0
$$
shows that $f$ is integral over ${\mathcal S}(N)$.

\vspace{1mm}

\noindent (2) Let $M:=(a,b)\subset\R$ be an open interval, where $a\in\R\cup\{-\infty\}$ and $b\in\R\cup\{\infty\}$, and let $\pi:M\to\R$ be a non constant closed Nash function. Let $\pi':M\to\R$ be its derivative, whose zeroset $\Zz_M(\pi')$ is finite. Otherwise it would be an $1$-dimensional semialgebraic set, and so it should contain an open subset $V$ of $M$. By the Identity Principle $\pi'\equiv0$, i.e., $\pi$ would be constant, and this is false. Let $\Zz_M(\pi'):=\{a_1,\dots,a_{\ell}\}$ and consider the closed subsets of $M$
$$
M_0:=(a,a_1],\quad M_1:=[a_1,a_2],\dots, M_{\ell-1}:=[a_{\ell-1},a_{\ell}]\quad\&\quad M_{\ell}:=[a_{\ell},b).
$$
For $0\leq i\leq {\ell}$ the restriction $\pi|_{M_i}$ is monotone, hence injective. Thus, it follows from part (1) that the induced homomorphism $\varphi_{\pi}:{\mathcal S}(\R)\to{\mathcal S}(M)$ is integral. This applies, for instance, to non constant polynomial maps $\pi:\R\to\R$, since they are closed Nash maps.

\vspace{1mm}

\noindent (3) The Nash function $\pi:\R\to\R,\, t\mapsto\sqrt{1+t^2}-t$ is not a closed map because $\pi(\R)=(0,+\infty)$ is not a closed subset of $\R$. Thus, by Proposition \ref{integral}, the induced homomorphism $\varphi_{\pi}$ is not integral.
\end{examples}

Before proving that finite homomorphisms between rings of semialgebraic functions are induced by locally injective semialgebraic maps we formulate Nakayama's Lemma as we will need for our purposes.

\begin{lem}\em (\em {\bf Nakayama}\em)\em\label{nak} Let $(A,\gtm_A)$ and $(B,\gtm_B)$ be local rings and let $\psi:A\to B$ be a finite homomorphism of local rings such that the homomorphism $\ol{\psi}:A/\gtm_A\to B/\gtm_A B$ defined by $\ol{\psi}(a+\gtm_A)=\psi(a)+\gtm_A B$ is surjective. Then, $\psi$ is surjective too.
\end{lem}

\begin{proof} It suffices to prove that $B=\psi(\gtm_A)B+\im\psi$ and apply next \cite[2.7]{am}, because $\gtm_A$ is the Jacobson radical of $A$. Given $b\in B$ there exists $a\in A$ such that 
$$
b+\psi(\gtm_A) B=\psi(a)+\psi(\gtm_A)B,
$$ 
i.e., $b-\psi(a)\in\psi(\gtm_A) B$, as wanted.
\end{proof}

\begin{thm} \label{lciny}  Let $\pi:M\to N$ be a semialgebraic map between semialgebraic subsets $M$ and $N$ such that the ring homomorphism $\varphi_{\pi}:{\mathcal S}(N)\to{\mathcal S}(M)$ is finite. Then the maps 
$$
\Specs(\pi):\Specs(M)\to\Specs(N),\quad\betas\pi:\betas M\to\betas N\quad\&\quad\pi:M\to N
$$
are proper, separated, locally injective and their fibers are finite sets.
\end{thm}
\begin{proof} As we observed in the proof of Proposition \ref{integral}, and since the homomorphism $\varphi_{\pi}$ is integral, the map $\Specs(\pi):\Specs(M)\to\Specs(N)$ is proper and separated. The finiteness of its fibers follows from \cite[Prop. 11]{s3}, and the properness of $\pi$ and $\betas\pi$ has been proved in Proposition \ref{integral}. The fibers of $\pi$ and $\betas\pi$ are subsets of the fibers of $\Specs(\pi)$, so they are finite too. In addition they are  Hausdorff spaces since $M$ and $\betas M$ are so.

To prove that $\Specs(\pi)$, $\betas\pi$ and $\pi$ are locally injective maps it suffices to demonstrate the following:

\paragraph{}\label{INJ} For every $\gtm\in\betas M$ there exists $f\in{\mathcal S}(M)$ such that $\gtm\in\Uu:=\Dd_{\Specs(M)}(f)$ and the restriction $\Specs(\pi)|_{\Uu}:\Uu\to\Specs(N)$ is injective. Once this is proved, every ideal $\gtp\in\Specs(M)$ is contained in a maximal ideal $\gtm$ of ${\mathcal S}(M)$, so $f\notin\gtp$, that is, $\Uu$ is an open neighbourhood of $\gtp$ in $\Specs(M)$, which proves that $\Specs(\pi)$ is locally injective. In addition, also the maps
$$
\betas\pi|_{\Uu\cap\betas M}=\Specs(\pi)|_{\Uu\cap\betas M}:\Uu\cap\betas M\to\betas N\quad\&\quad\pi|_{\Uu\cap M}=\betas\pi|_{\Uu\cap M}:\Uu\cap M\to N 
$$
are injective, so both $\betas\pi$ and $\pi$ are locally injective.

Let us prove \ref{INJ}. Denote $\gtn:=\Specs(M)(\gtm)\in\betas N$, and consider the induced homomorphism of local rings 
$$
\psi:A:={\mathcal S}(N)_{\gtn}\to B:={\mathcal S}(M)_{\gtm},\, \frac{u}{v}\mapsto\frac{\varphi_{\pi}(u)}{\varphi_{\pi}(v)},
$$
which is finite since $\varphi_{\pi}$ is so. We are going to prove that $\psi$ is surjective. Let $\gtm_A:=\gtn A$ and $\gtm_B:=\gtm B$ be, respectively, the unique maximal ideals in the local rings $A$ and $B$. Then $\gtm_A B\subset\gtm_B$, and we will prove right now that $\sqrt{\gtm_A B}=\gtm_B$. To that end it suffices to see that $\gtm_B$ is the unique prime ideal of $B$ that contains $\gtm_AB$. Let $\gtp$ be a prime ideal in $B$ that contains $\gtm_AB$. Then $\gtp\subset\gtm_B$ and, using \cite[5.9]{am}, to prove that $\gtp=\gtm_B$ it is enough to see that both prime ideals lie over $\gtm_A$ via the integral homomorphism $\psi$. But $\gtm_A\subset\psi^{-1}(\psi(\gtm_A)B)\subset\psi^{-1}(\gtp)$ and, since $\gtm_A$ is maximal, we have $\gtm_A=\psi^{-1}(\gtp)$, whereas $\gtm_A=\psi^{-1}(\gtm_B)$ because $\gtn=\varphi_{\pi}^{-1}(\gtm)$. Consequently, 
\begin{equation}\label{keyeq1}
\begin{split}
\gtm=\gtm_B\cap{\mathcal S}(M)=\sqrt{\gtm_AB}\cap{\mathcal S}(M)\subset\sqrt{\gtm_AB\cap{\mathcal S}(M)}=\sqrt{\gtn B\cap{\mathcal S}(M)}.
\end{split}
\end{equation}
This implies $\gtm=\sqrt{\gtn B\cap{\mathcal S}(M)}$ because $\gtm$ is maximal. We will show now that 
\begin{equation}\label{keyeq}
\begin{split}
\gtm=\gtn B\cap{\mathcal S}(M).
\end{split}
\end{equation}
To prove it note that, by the finiteness of $\varphi_{\pi}$, there exist $f_1,\dots,f_r\in{\mathcal S}(M)$ such that
\begin{equation}\label{keyeq0}
\begin{split}
{\mathcal S}(M)=f_1{\mathcal S}(N)+\cdots+f_r{\mathcal S}(N).
\end{split}
\end{equation}
On the other hand the field homomorphism
$$
\phi:\kappa(\gtn):={\mathcal S}(N)/\gtn\to\kappa(\gtm):={\mathcal S}(M)/\gtm,\, g+\gtn\mapsto(g\circ\pi)+\gtm
$$
is finite because $\varphi_{\pi}$ is so. Thus $\kappa(\gtm)$ is an algebraic extension of $\kappa(\gtn)$ and both are real closed fields, which implies that $\phi$ is surjective. Hence, for every $1\leq j\leq r$ there exist $g_j\in{\mathcal S}(N)$ such that $f_j+\gtm=(g_j\circ\pi)+\gtm$, so $h_j:=f_j-(g_j\circ\pi)\in\gtm$. 

By equality \eqref{keyeq1} there exists a positive integer $n\in\Z$ such that $h_j^n\in\gtn B\cap{\mathcal S}(M)$ for $1\leq j\leq r$. On the other hand, for each function $f\in{\mathcal S}(M)$ there exist $\ell_1,\dots,\ell_r\in{\mathcal S}(N)$ such that
$$
f=\sum_{j=1}^rf_j\cdot(\ell_{j}\circ\pi)=\sum_{j=1}^r(h_j+(g_j\circ\pi))\cdot(\ell_{j}\circ\pi)=\sum_{j=1}^rh_j\cdot(\ell_{j}\circ\pi)+\sum_{j=1}^rg_j\ell_j\circ\pi.
$$
In other words,
\begin{equation}\label{keyeq2}
\begin{split}
{\mathcal S}(M)=h_1{\mathcal S}(N)+\cdots+h_r{\mathcal S}(N)+h_{r+1}{\mathcal S}(N),
\end{split}
\end{equation}
where $h_{r+1}={\bf 1}_M$. 

Let $m>n(r+1)$ be an odd positive integer. We claim that $h^m\in\gtn B\cap{\mathcal S}(M)$ for every $h\in\gtm$. To see this observe that there exist $q_1,\dots,q_{r+1}\in{\mathcal S}(N)$ such that 
$$
h=\sum_{j=1}^rh_j\cdot(q_j\circ\pi)+(q_{r+1}\circ\pi),
$$
hence
$$
\varphi_{\pi}(q_{r+1})=q_{r+1}\circ\pi=h-\sum_{j=1}^rh_j\cdot(q_j\circ\pi)\in\gtm,
$$
that is, $q_{r+1}\in\varphi_{\pi}^{-1}(\gtm)=\gtn$. Thus, 
$$
h^{m}=\Big(\sum_{j=1}^rh_j\cdot(q_j\circ\pi)+(q_{r+1}\circ\pi)\Big)^m\in\gtn B\cap{\mathcal S}(M)
$$
because it is a sum whose summands contain either a factor $h_j^{r_j}$ with $r_j\geq n$ or a power of $q_{r+1}\circ\pi$.

We are ready to check equality \eqref{keyeq}, where the inclusion $\gtn B\cap{\mathcal S}(M)\subset\gtm$ is evident. Conversely, let $f\in\gtm$. Since $m$ is odd there exists $h\in{\mathcal S}(M)$ such that $h^m=f\in\gtm$. Thus $h\in\gtm$ and, as we have just proved, this implies that $f=h^m\in \gtn B\cap{\mathcal S}(M)$. Consequently $\gtm_AB=\gtm_B$ because
$$
\gtm_B=\gtm B=(\gtn B\cap{\mathcal S}(M))B\subset\gtn B=\gtm_AB\subset\gtm_B.
$$
Thus, the homomorphism $\psi:A\to B$ is surjective since it satisfies the hypothesis in Nakayama's Lemma \ref{nak}; $\psi$ induces a surjective homomorphism
$$
\phi:=\ol{\psi}:A/\gtm_A=\kappa(\gtn)\to B/\gtm_A B= B/\gtm B=\kappa(\gtm).
$$
In this way we have a commutative square
$$
\xymatrix{
{\mathcal S}(N)\ar[d]\ar@{->}[r]^{\hspace{3mm}\varphi_{\pi}\quad}&{\mathcal S}(M)\ar[d]\\
{\mathcal S}(N)_{\gtn}\ar@{->}[r]^{\hspace{3mm}\psi\quad}&{\mathcal S}(M)_{\gtm}
}
$$
where $\varphi_{\pi}$ is finite and $\psi$ is surjective. Let $f_1,\dots,f_r\in{\mathcal S}(M)$ satisfying equality \eqref{keyeq0}. As $\psi$ is surjective there exist $h_1,\dots,h_r\in{\mathcal S}(N)$ and $g_1,\dots,g_r\in{\mathcal S}(N)\setminus\gtn$ such that, for $1\leq i\leq r$, 
$$
f_i=\frac{h_i\circ\pi}{g_i\circ\pi}.
$$  
Denote $p_i:=h_i\cdot\prod_{j\neq i}g_j\in{\mathcal S}(N)$. Then 
$$
g:=\prod_{i=1}^rg_i\in{\mathcal S}(N)\setminus\gtn\quad\&\quad f_i=\frac{p_i\circ\pi}{g\circ\pi}.
$$
Thus $f:=g\circ\pi\in{\mathcal S}(M)\setminus\gtm$ and we claim that the homomorphism
$$
\widetilde{\varphi}_{\pi}:{\mathcal S}(N)_g\to{\mathcal S}(M)_f,\, \frac{u}{g^n}\mapsto \frac{u\circ\pi}{f^n}
$$
is surjective. Indeed, given $\zeta\in{\mathcal S}(M)_f$ there exist $v\in{\mathcal S}(M)$ and a positive integer $k$ such that $\zeta:=\frac{v}{f^k}$. Write 
$$
v=f_1(u_1\circ\pi)+\cdots+f_r(u_r\circ\pi), 
$$
for some $u_1,\dots,u_r\in{\mathcal S}(N)$. Then
$$
\zeta=\frac{v}{f^k}=\frac{\sum_{i=1}^rf_i(u_i\circ\pi)}{f^k}=\frac{\sum_{i=1}^r(p_i\circ\pi)(u_i\circ\pi)}{(g\circ\pi)^{k+1}}.
$$
Therefore, 
$$
\zeta':=\frac{\sum_{i=1}^rp_iu_i}{g^{k+1}}\in{\mathcal S}(N)_g\quad\text{and}\quad\widetilde{\varphi}_{\pi}(\zeta')=\zeta. 
$$
Finally, $\Uu:=\Dd_{\Specs(M)}(f)=\Spec({\mathcal S}(M)_f)$ is a neighbourhood of $\gtm$ in $\Specs(M)$ and the map 
$$
\Specs(\pi)|_{\Uu}=\Spec(\widetilde{\varphi}_{\pi}):\Uu\to\Spec({\mathcal S}(N)_g)\hookrightarrow\Specs(N)
$$
is injective.  
\end{proof}

\begin{remark}\label{FNTSS} \em It follows from the proof of the previous Theorem \ref{lciny} that given a semialgebraic map $\pi:M\to N$ such that the induced homomorphism $\varphi_{\pi}:{\mathcal S}(N)\to{\mathcal S}(M)$ is finite, there exists a subset $\{f_i:\, i\in I\}\subset{\mathcal S}(M)$ such that 
$$
\betas M=\bigcup_{i\in I}\Dd_{\betas M}(f_i)\quad\&\quad\betas\pi|_{\Dd_{\betas M}(f_i)}:\Dd_{\betas M}(f_i)\to\betas N\ \text{ is injective for each }\ i\in I.
$$
In addition we may assume that $I$ is finite because $\betas M$ is compact. Thus $\{\Dd_M(f_i):\, i\in I\}$ is a finite cover by open semialgebraic subsets of $M$ and each restriction $\pi|_{\Dd_M(f_i)}:\Dd_M(f_i)\to N$ is injective.
\end{remark}

\begin{cor}\label{chliy} Let $M\subset\R^m$ and $N\subset\R^n$ be compact semialgebraic sets. Then, a semialgebraic map $\pi:M\to N$ is locally injective if and only if the induced homomorphism $\varphi_{\pi}:{\mathcal S}(N)\to{\mathcal S}(M)$ is finite.
\end{cor}
\begin{proof} Suppose first that $\pi$ is locally injective. For each point $x\in M$ there exists an open ball $\Bb_x\subset\R^m$ centered at $x$ such that the restriction $\pi|_{\cl_M(\Bb_x)}:{\cl_M(\Bb_x)}\to N$ is injective. As $\cl_M(\Bb_x)$ is compact this implies that $\pi|_{\cl_M(\Bb_x)}$ is a semialgebraic homeomorphism onto its image $\pi(\cl_M(\Bb_x))$. Hence, the ring homomorphism 
$$
{\mathcal S}(\pi(\cl_M(\Bb_x)))\to{\mathcal S}(\cl_M(\Bb_x)),\,f\mapsto  f\circ\pi
$$
is an isomorphism. On the other hand, by Lemma \ref{tietze} (2), the restriction homomorphism ${\mathcal S}(N)\to{\mathcal S}(\pi(\cl_M(\Bb_x)))$ is surjective. Therefore, the map
$$
{\mathcal S}(N)\to{\mathcal S}(\cl_M(\Bb_x)),\, f\mapsto f|_{\pi(\cl_M(\Bb_x))}\circ\pi|_{\cl_M(\Bb_x)}
$$ 
is surjective too. Since $M$ is compact there exist finitely many points $x_1,\dots,x_r\in M$ such that $M\subset\bigcup_{i=1}^r\Bb_{x_i}$. For $1\leq i\leq r$ the difference $M\setminus\Bb_{x_i}$ is a closed semialgebraic subset of $M$ and, by Lemma \ref{tietze} (1), there exists $f_i\in{\mathcal S}(M)$ with $M\setminus\Bb_{x_i}=\Zz_M(f_i)$. In fact, changing $f_i$ by $f_i^2$ we may assume that $f_i(x)\geq0$ for every point $x\in M$. This together with the inclusion $M\subset\bigcup_{i=1}^r\Bb_{x_i}$ implies that $\Zz_M(f)=\varnothing$, where $f:=\sum_{i=1}^rf_i$. Consequently each quotient $g_i:=\frac{f_i}{f}\in{\mathcal S}(M)$ and all reduces to prove the equality
$$
{\mathcal S}(M)=g_1{\mathcal S}(N)+\cdots+g_r{\mathcal S}(N).
$$
Given $h\in{\mathcal S}(M)$ and $1\leq i\leq r$, its restriction $h|_{\cl_M(\Bb_{x_i})}\in{\mathcal S}(\cl_M(\Bb_{x_i}))$ and there exists $u_i\in{\mathcal S}(N)$ such that $u_i|_{\pi(\cl_M(\Bb_{x_i}))}\circ\pi|_{\cl_M(\Bb_{x_i})}=h|_{\cl_M(\Bb_{x_i})}$. Therefore,
$$
h\cdot f_i=f_i\cdot(u_i\circ\pi)\ \text{ for }\ 1\leq i\leq r,
$$
which implies
$$
h\cdot f=h\cdot\sum_{i=1}^rf_i=\sum_{i=1}^rh\cdot f_i=\sum_{i=1}^rf_i\cdot(u_i\circ\pi).
$$
Dividing both members by $f$, that is a unit in ${\mathcal S}(M)$, and since $\frac{f_i}{f}=g_i$ we get
$$
h=\sum_{i=1}^r\left(\frac{f_i}{f}\right)\cdot(u_i\circ\pi)=\sum_{i=1}^rg_i\cdot(u_i\circ\pi)\in g_1{\mathcal S}(N)+\cdots+g_r{\mathcal S}(N).
$$
The converse follows straightforwardly from Theorem \ref{lciny}.
\end{proof}

\begin{examples} \em (1) Let $\sph^n\subset\R^{n+1}$ be the unit sphere and let $\Proy^n(\R)$ be the $n$-dimensional real projective space. Both are compact real algebraic sets, hence semialgebraic; for the case of the projective space see \cite[3.4.4]{bcr}, where it is proved that $\Proy^n(\R)$ can be understood as a real algebraic subset of $\R^{n^2}$. As the canonical projection
$$
\pi:\sph^n\to\Proy^n(\R),\, x:=(x_0,\dots,x_n)\mapsto[x_0:\cdots:x_n]
$$
is a locally injective semialgebraic map the homomorphism $\varphi_{\pi}:{\mathcal S(\Proy^n(\R))}\to{\mathcal S(\sph^n)}$ is, by Corollary \ref{chliy}, finite. In particular it is integral and finitely generated. However it is not simple. Otherwise there would exist $h\in{\mathcal S(\sph^n)}$ such that ${\mathcal S(\sph^n)}={\mathcal S(\Proy^n(\R))}[h]$. By Borsuk-Ulam theorem, \cite[68.6]{mu} there exists a point $p\in\sph^n$ such that $h(p)=h(-p)$. Consider the semialgebraic function
$$
g:\sph^n\to\R,\, x\mapsto\|x-p\|.
$$
Then there exist a positive integer $d$ and functions $f_0,\dots,f_d\in{\mathcal S(\Proy^n(\R))}$ such that 
$$
g=(f_0\circ\pi)+(f_1\circ\pi)h+\cdots+(f_d\circ\pi)h^d.
$$
Since $\pi(-p)=\pi(p)$ and $g(-p)=2\|p\|\neq0$ whereas $g(p)=0$, we get a contradiction:
$$
0\neq g(-p)=\sum_{i=0}^d(f_i\circ\pi)(-p)h^i(-p)=\sum_{i=0}^d(f_i\circ\pi)(p)h^i(p)=g(p)=0.
$$

\noindent (2) Let ${\tt i}:=\sqrt{-1}$ and $\tau:\C\to\R^2,\, z:=x+{\tt i}y\mapsto (x,y)$. Fix a positive integer $n$ and consider the analytic map $\phi:\C\to\C, \, z\mapsto z^n$. The semialgebraic map
$$
\pi:=\tau\circ\phi\circ\tau^{-1}:M:=\R^2\to N:=\R^2
$$
is not locally injective at the origin, and it follows from Theorem \ref{lciny} that the induced homomorphism $\varphi_{\pi}:{\mathcal S}(N)\to{\mathcal S}(M)$ is not finite. However $\varphi_{\pi}$ is an integral homomorphism. To prove this, let $f\in{\mathcal S}(M)$ and consider the elementary symmetric forms in $n$ variables
$$
\sigma_j({\tt x}_1,\dots,{\tt x}_n):=(-1)^j\sum_{1\leq i_1<\cdots<i_j\leq n}{\tt x}_{i_1}\cdots{\tt x}_{i_j}\in\Z[{\tt x}_1,\dots,{\tt x}_n].
$$
For every point $z\in N\setminus\{(0,0)\}$ let $\pi^{-1}(z):=\{z_1,\dots,z_n\}\subset M$. For $1\leq j\leq n$ define the function
$$
S_j(f):N\setminus\{(0,0)\}\to\R,\, z\mapsto \sigma_j(f(z_{i_1}),\dots,f(z_{i_n})).
$$
Since $(S_j(f)\circ\tau)|_{\C\setminus\{0\}}:\C\setminus\{0\}\to\C$ is a locally bounded analytic function, $S_j(f)$ admits, by  \cite[12.3]{gpr}, a continuous extension that we also denote $S_j(f):N\to\R$. This function is semialgebraic, see the Appendix in \cite{bfg}. Hence, 
$$
{\tt P}({\tt t}):={\tt t}^n+\sum_{j=1}^n S_j(f){\tt t}^{n-j}\in{\mathcal S}(N)[\tt t]
$$ 
is a monic polynomial and ${\tt P}(f)=0$, or, more precisely, 
$$
f^n+\sum_{j=1}^n(S_j(f)\circ\pi)f^{n-j}=0,
$$
which proves that $f$ is integral over ${\mathcal S}(N)$.
\end{examples}

\begin{lem} \label{cnyrz} Let $M\subset\R^m$ be a semialgebraic set and let $f_1,\dots,f_n\in{\mathcal S}(M)$. Consider the polynomial
$$
{\tt p}({\tt t}):={\tt t}^n+\sum_{j=1}^nf_j{\tt t}^{n-j}\in{\mathcal S}(M)[{\tt t}],
$$
and suppose that for each point $x\in M$ the polynomial 
$$
{\tt p}_x({\tt t}):={\tt t}^n+\sum_{j=1}^nf_j(x){\tt t}^{n-j}\in\R[{\tt t}]
$$
splits in $\R[{\tt t}]$ as a product of \em(\em non necessarily distinct\,\em) \em factors of degree $1$. Then, there exist $g_1,\dots, g_n\in{\mathcal S}(M)$ such that
$$
{\tt p}({\tt t})=\prod_{j=1}^n({\tt t}-g_j).
$$ 
\end{lem}
\begin{proof} For every point $u:=(u_1,\dots,u_n)\in\R^n$ let $\zeta_1(u),\dots,\zeta_n(u)$ be the real parts of the complex zeros of the polynomial
$$
{\tt q}_u({\tt t}):={\tt t}^n+\sum_{j=1}^nu_j{\tt t}^{n-j}\in\R{[\tt t}]
$$
listing each according to its multiplicity and indexed so that $\zeta_1(u)\leq\cdots\leq\zeta_n(u)$. It was proved in the Appendix of \cite{bfg} that each $\zeta_j\in{\mathcal S}(\R^n)$. 

Now notice that, for every point $x\in M$, the roots in $\C$ of the polynomial ${\tt p}_x({\tt t})$ are the real numbers $\zeta_j(f_j(x))$. Since ${\tt p}_x$ is a monic polynomial this means that 
$$
{\tt p}_x({\tt t})=\prod_{j=1}^n({\tt t}-\zeta_j(f_j(x)))\ \text{ for every point } x\in M.
$$
Hence $g_j:=\zeta_j\circ f_j\in{\mathcal S}(M)$ for $1\leq j\leq n$ and ${\tt p}({\tt t})=\prod_{j=1}^n({\tt t}-g_j)$. 
\end{proof}

\begin{cor} \label{depent} Let $\pi:M\to N$ be a surjective semialgebraic map and let $g\in{\mathcal S}(M)$ be integral over ${\mathcal S}(N)$ via the induced homomorphism $\varphi_{\pi}$. Then, there exist an integer $k>0$ and $h_1,\dots, h_k\in{\mathcal S}(N)$ such that $\prod_{j=1}^k(g-h_j\circ\pi)=0$.
\end{cor}
\begin{proof} Since $g$ is integral over ${\mathcal S}(N)$ there exist an integer $k>0$ and $f_1,\dots,f_k\in{\mathcal S}(N)$ such that
$$
g^k+(f_1\circ\pi)g^{k-1}+\cdots+(f_{k-1}\circ\pi)g+(f_k\circ\pi)=0.
$$
Consider the semialgebraic map 
$$
\phi:N\to\R^n,\, y\mapsto (f_1(y),\dots, f_k(y)).
$$ 
Then, with the notations in Lemma \ref{cnyrz}, $g(x)\in\{\zeta_j(\phi(\pi(x))):\, 1\leq j\leq k\}$ for each point $x\in M$ or, equivalently, the functions $h_j:=\zeta_j\circ\phi\in{\mathcal S}(N)$ satisfy the equality in the statement.
\end{proof}

\begin{example} \em (1) Let $\pi:M\to N$ be a surjective semialgebraic map and let 
$g\in{\mathcal S}(M)$ be integral over ${\mathcal S}(N)$ via $\varphi_{\pi}$. Let  $\phi:\R\to\R$ be a semialgebraic function. Then $\phi\circ g$ is integral over ${\mathcal S}(N)$ via $\varphi_{\pi}$. 

Indeed, by the previous Corollary \ref{depent}, there exist an integer $k>0$ and semialgebraic functions $h_1,\dots, h_k\in{\mathcal S}(N)$ such that $\prod_{j=1}^k(g-h_j\circ\pi)=0$. Thus, for every point $x\in M$ there exists an index $j$ with $1\leq j\leq k$ such that $g(x)=h_j(\pi(x))$ and  
$$
(\phi\circ g)(x)=\phi(g(x))=\phi(h_j(\pi(x))).
$$
This means that $\prod_{j=1}^k(\phi\circ g-\phi\circ h_j\circ\pi)=0$ and $\phi\circ g$ is integral over ${\mathcal S}(N)$ via $\varphi_{\pi}$ because each $\phi\circ h_j\in{\mathcal S}(N)$.

\noindent (2) Choosing $\phi:\R\to\R,\, t\mapsto|t|$ it follows that $|g|$ is integral over ${\mathcal S}(N)$ via $\varphi_{\pi}$ for every function $g\in{\mathcal S}(M)$ that is integral over ${\mathcal S}(N)$ via $\varphi_{\pi}$.
\end{example}

\begin{defn} \em Let $M\subset\R^m$ be a semialgebraic set. A subset ${\mathcal F}\subset{\mathcal S}(M)$ \em separates points \em of $M$ if for every pair of distinct points $p,q\in M$ there exists a function $f\in{\mathcal S}(M)$ such that $f(p)\neq f(q)$.
\end{defn}

\begin{prop} \label{sep} Let $\pi:M\to N$ be a closed and surjective semialgebraic map between semialgebraic sets $M$ and $N$. Suppose that there exist $f_1,\dots,f_k\in{\mathcal S}(M)$ such that each $f_i$ is integral over ${\mathcal S}(N)$ via $\varphi_{\pi}$ and ${\mathcal S}(N)[f_1,\dots,f_k]$ separates points of $M$. Then, there exists a finite cover $\{M_1,\dots,M_r\}$ by closed semialgebraic subsets of $M$ such that each restriction $\pi|_{M_i}: M_i\to N$ is injective. In particular, $\varphi_{\pi}$ is an integral homomorphism.
\end{prop}
\begin{proof} Notice that the last part is the immediate consequence of the first one and Example \ref{lcinycv} (1). To prove the first one notice that, by the previous Corollary \ref{depent}, for $1\leq i\leq k$ there exist $g_{i1},\dots, g_{i\ell_{i}}\in{\mathcal S}(N)$ such that 
$$
\prod_{j=1}^{\ell_i}(f_i-g_{ij}\circ\pi)=0.
$$
Denote $h_{ij}:=f_i-g_{ij}\circ\pi\in{\mathcal S}(M)$ and $Z_{ij}:=\Zz_M(h_{ij})$ for $1\leq i\leq k$ and $1\leq j\leq\ell_i$. Each $Z_{ij}$ is a closed semialgebraic subset of $M$ and, for fixed $i$, we have $M=\bigcup_{j=1}^{\ell_i}Z_{ij}$. 

Let $\Sigma$ be the set of all $k$-tuples $J:=(j_1,\dots,j_k)$ with $1\leq j_i\leq\ell_i$ and let
$$
h_J:=\sum_{i=1}^kh_{ij_i}^2\in{\mathcal S}(M)\quad\&\quad M_J:=\Zz_M(h_J)=\bigcap_{i=1}^{k}\Zz_{M}(h_{ij_i}).
$$
Observe that for every point $x\in M$ and each index $i$ with $1\leq i\leq k$ there exists an index $j_i$ with $1\leq j_i\leq \ell_i$ such that $x\in Z_{i{j_i}}$. Thus $x\in M_J$, where  $J:=(j_1,\dots,j_k)$. In other words, the family $\{M_J:\, J\in\Sigma\}$ is a finite cover by closed semialgebraic subsets of $M$, and all reduces to check that each restriction $\pi|_{M_J}:M_J\to N$ is injective. Suppose, by the way of contradiction, that there exist two distinct points $p,q\in M_J$ with $\pi(p)=\pi(q)$. Since ${\mathcal S}(N)[f_1,\dots,f_k]$ separates points there exists a polynomial 
$$
{\tt p}({\tt x}_1,\dots,{\tt x}_k):=\sum u_{i_1,\dots,i_k}{\tt x}_{i_1}^{\nu_{i_1}}\cdots{\tt x}_{i_k}^{\nu_{i_k}}\in{\mathcal S}(N)[{\tt x}_1,\dots,{\tt x}_k]
$$
such that the function 
$$
f:={\tt p}(f_1,\dots,f_k)=\sum(u_{i_1,\dots,i_k}\circ\pi)f_{i_1}^{\nu_{i_1}}\cdots f_{i_k}^{\nu_{i_k}}\in{\mathcal S}(M)
$$
satisfies $f(p)\neq f(q)$. But for each multiindex $(i_1,\dots,i_k)$ one has
$$
(u_{i_1,\dots,i_k}\circ\pi)(p)=u_{i_1,\dots,i_k}(\pi(p))=u_{i_1,\dots,i_k}(\pi(q))=(u_{i_1,\dots,i_k}\circ\pi)(q),
$$
so there exists some index $i$ with $1\leq i\leq k$ such that $f_i(p)\neq f_i(q)$. Since $p,q\in M_J$ we know that $h_{ij_i}(p)=h_{ij_i}(q)=0$, that is, 
$$
f_i(p)=g_{ij_i}(\pi(p))=g_{ij_i}(\pi(q))=f_i(q),
$$
a contradiction.
\end{proof}

\begin{cor} Let $\pi:M\to N$ be a closed and surjective semialgebraic map between semialgebraic sets $M$ and $N$ and suppose that there exists $f\in{\mathcal S}(M)$ that is integral over ${\mathcal S}(M)$ via $\varphi_{\pi}$ and it is injective on each fiber of $\pi$. Then, there exists a finite cover $\{M_1,\dots,M_r\}$ by closed semialgebraic subsets of $M$ such that each map $\pi|_{M_i}: M_i\to N$ is injective. In particular, $\varphi_{\pi}$ is an integral homomorphism.
\end{cor}

\begin{proof} By using Proposition \ref{sep} it is enough to check that ${\mathcal S}(M)[f]$ separates points in $M$. Indeed, let $p,q\in M$ be two distinct points. If $\pi(p)\neq\pi(q)$ the semialgebraic function
$$
g:=\dist(.,\pi(p)):N\to\R,\, y\mapsto\dist(y,\pi(p))
$$
vanishes at $\pi(p)$ and $g(\pi(q))\neq0$. Thus, $(g\circ\pi)(p)\neq(g\circ\pi)(q)$. On the other hand, if $\pi(p)=\pi(q):=y$ then $f(p)\neq f(q)$ because $f|_{\pi^{-1}(y)}$ is injective.
\end{proof}

\begin{cor} Let $\pi:M\to N$ be a closed and surjective semialgebraic map between semialgebraic sets $M$ and $N$. Then, the induced homomorphism $\varphi_{\pi}:{\mathcal S}(N)\to{\mathcal S}(M)$ is integral if and only if there exists a finite cover $\{M_i:\, 1\leq i\leq k\}$ by closed semialgebraic subsets of $M$ such that each restriction $\pi|_{M_i}:M_i\to N$ is injective.
\end{cor}
\begin{proof} The ``if part'' was proved in Example  \ref{lcinycv} (1). For the ``only if'' part suppose that $M\subset\R^k$ and choose
$$
f_i:M\to\R,\, x:=(x_1,\dots,x_m)\mapsto x_i,\quad\text{for } i=1,\dots,k.
$$
As $f_1,\dots,f_k\in{\mathcal S}(M)$ and ${\mathcal S}(N)[f_1,\dots,f_k]$ separates points of $M$, the conclusion follows from Proposition \ref{sep}.
\end{proof}

\begin{thm} Let $\pi:M\to N$ be a surjective semialgebraic map between the compact semialgebraic sets $M$ and $N$ such that the induced homomorphism  $\varphi_{\pi}:{\mathcal S}(N)\to{\mathcal S}(M)$ is simple. Then $\pi$ is locally injective.
\end{thm}
\begin{proof} By Corollary \ref{chliy} it is enough to prove that $\varphi_{\pi}$ is finite. By the hypothesis there exists a function $g\in{\mathcal S}(M)$ such that ${\mathcal S}(M)={\mathcal S}(N)[g]$ and, using \cite[5.1]{am}, it suffices to see that $g$ is integral over ${\mathcal S}(N)$. We may assume that $g(x)>0$ for every $x\in M$. Indeed, since $M$ is compact there exists $r:=\min\,\{g(x):\, x\in M\}$, and the function $g-r+1$ is positive on $M$ and ${\mathcal S}(N)[g]={\mathcal S}(N)[g-r+1]$. In addition, using again the compactness of $M$, there exists $\veps>0$ such that $g(x)<\veps$ for each $x\in M$. Since $\Zz_M(g)=\varnothing$ also the quotient $\frac{1}{g}\in{\mathcal S}(M)={\mathcal S}(N)[g]$, and there exist an integer $n\geq0$ and $f_0,\dots,f_n\in{\mathcal S}(N)$ with
$$
\frac{1}{g}=(f_0\circ\pi)+(f_1\circ\pi)g+(f_2\circ\pi)g^2+\cdots+(f_{n-1}\circ\pi)g^{n-1}+(f_n\circ\pi)g^n.
$$
Dividing both members by $g^n$ we get
$$
\frac{1}{g^{n+1}}=(f_0\circ\pi)\left(\frac{1}{g}\right)^n+(f_1\circ\pi)\left(\frac{1}{g}\right)^{n-1}+(f_2\circ\pi)\left(\frac{1}{g}\right)^{n-2}+\cdots+(f_{n-1}\circ\pi)\left(\frac{1}{g}\right)+(f_n\circ\pi).
$$
Therefore $\frac{1}{g}$ is integral over ${\mathcal S}(N)$ and, by Corollary \ref{depent}, there exist $h_0,\dots, h_n\in{\mathcal S}(N)$ satisfying 
$$
\prod_{j=0}^n\Big(\frac{1}{g}-h_j\circ\pi\Big)=0.
$$
Hence, for each $x\in M$ there exists $j$ with $0\leq j\leq n$ such that $h_j(\pi(x))=\frac{1}{g(x)}$. The function $\ell_j:=\max\,\{h_j,\frac{1}{\veps}\}\in{\mathcal S}(N)$ has empty zeroset and $\frac{1}{h(\pi(x))}=g(x)<\veps$. Thus $h_j(\pi(x))>\frac{1}{\veps}$, that is, $\ell_j(\pi(x))=\frac{1}{g(x)}$ or, equivalently, $g(x)=\frac{1}{\ell_j(x)}$. Consequently, $g$ is integral over ${\mathcal S}(N)$, because
$$
\prod_{j=0}^n\Big(g-\frac{1}{\ell_j\circ\pi}\Big)=0.
$$
\end{proof}

Data availability statement: arXiv. submit/6277125 [math.AG] 13 Mar.2025

\bibliographystyle{amsalpha}

\end{document}